\documentclass[12pt]{amsart}
\usepackage{fullpage}

\usepackage{setspace}
\RequirePackage{amssymb}
\usepackage{amsbsy}
\usepackage{enumerate}
\usepackage{mathrsfs}
\usepackage{undertilde}
\usepackage{amsmath}
\usepackage{url}

\theoremstyle{plain}
\newtheorem{theorem}{Theorem}[section]
\newtheorem*{theorem*}{Theorem}
\newtheorem*{mainthm*}{Main Theorem}
\newtheorem{lemma}[theorem]{Lemma} \newtheorem*{lem*}{Lemma}
\newtheorem{claim}[theorem]{Claim} \newtheorem*{claim*}{Claim}
 \newtheorem*{cor*}{Corollary}
 \newtheorem*{prop*}{Proposition}

\theoremstyle{definition}
\newtheorem{definition}[theorem]{Definition} \newtheorem*{defn*}{Definition}

\theoremstyle{remark}
\newtheorem{remark}[theorem]{Remark} \newtheorem*{rem*}{Remark}
 \newtheorem*{example*}{Example}
\newtheorem{question}[theorem]{Question} \newtheorem*{question*}{Question}

\newcommand{\Ord}{\mathrm{Ord}}

\newcommand{\forces}{\mathbin{\Vdash}}
\newcommand{\bfSigma}{\utilde{\bf{\Sigma}}}
\newcommand{\bfPi}{\utilde{\bf{\Pi}}}
\newcommand{\bfDelta}{\utilde{\bf{\Delta}}}
\DeclareMathOperator{\powerset}{\mathcal{P}}

\DeclareMathOperator{\Col}{Col}
\DeclareMathOperator{\crit}{crit}

\DeclareMathOperator{\rank}{rank}

\begin{document}

\title{On forcing projective generic absoluteness from strong cardinals}\thanks{Part of this research was conducted at the University of California, Irvine, where the author was supported by the National Science Foundation grant DMS-1044150}

\author{Trevor M. Wilson}

\address{Department of Mathematics\\Miami University\\Oxford, Ohio 45056\\USA}
\email{twilson@miamioh.edu} 
\urladdr{https://www.users.miamioh.edu/wilso240}

\begin{abstract}
 W.H.~Woodin showed that if $\kappa_1 < \cdots < \kappa_n$ are strong cardinals then two-step $\bfSigma^1_{n+3}$ generic absoluteness holds after collapsing $2^{2^{\kappa_n}}$ to be countable.  We show that this number can be reduced to $2^{\kappa_n}$, and to $\kappa_n^+$ in the case $n = 1$, but cannot be further reduced~to~$\kappa_n$.
\end{abstract}

\maketitle

\section{Introduction}

The goal of this article is to slightly sharpen a theorem of W.H.~Woodin regarding forcing generic absoluteness from strong cardinals. We begin with a brief introduction to generic absoluteness principles and a summary of their relationship to large cardinals.
For a positive integer $n$, \emph{one-step $\bfSigma^1_n$ generic absoluteness} says that every $\bfSigma^1_n$ statement is absolute between $V$ and all generic extensions, and \emph{two-step $\bfSigma^1_n$ generic absoluteness} says that for every generic extension of $V$, every $\bfSigma^1_n$ statement is absolute between it and all further generic extensions. Crucially, two-step generic absoluteness allows real parameters from generic extensions. \emph{Projective generic absoluteness} (either one-step or two-step) says that $\bfSigma^1_n$ generic absoluteness holds for every positive integer $n$.

From Shoenfield's absoluteness theorem (see, for example, Kanamori \cite[Theorem 13.15]{KanHigherInfinite}) it follows that one-step $\bfSigma^1_2$ generic absoluteness is a theorem of ZFC and therefore two-step $\bfSigma^1_2$ generic absoluteness is also a theorem of ZFC. Generic absoluteness principles for  higher pointclasses in the projective hierarchy are independent of ZFC and are equiconsistent with the existence of large cardinals. For example, one-step $\bfSigma^1_3$ generic absoluteness is equiconsistent with the existence of a $\Sigma_2$-reflecting cardinal (see Feng, Magidor, and Woodin \cite[Corollary 3.1 and Theorem 3.3]{FenMagWoo}.)

Two-step $\bfSigma^1_3$ generic absoluteness has higher consistency strength than one-step $\bfSigma^1_3$ generic absoluteness: it is equivalent to the statement ``every set has a sharp''. Using Jensen's covering lemma, Woodin \cite{WooProjUnif} showed that if two-step $\bfSigma^1_3$ generic absoluteness holds then every set has a sharp. The converse implication originates with Martin and Solovay \cite[Lemma 5.6]{MarSolBasis}; see Kechris \cite[Section 2]{KecHomogeneousTrees} or Kanamori \cite[Section 15]{KanHigherInfinite} for a modern treatment of the subject. (A simpler proof of $\bfSigma^1_3$ generic absoluteness from sharps is given by Caicedo and Schindler \cite[Theorem 3]{CaiSchProjWellord}. However, elements of the Martin--Solovay argument will be needed later in this paper.)

At higher levels of complexity, generic absoluteness relates to strong cardinals. By results of Woodin and K.~Hauser, one-step and two-step projective generic absoluteness are both equiconsistent with the existence of infinitely many strong cardinals. More specifically, Woodin showed that if $\lambda$ is a limit of strong cardinals, then two-step projective generic absoluteness holds after forcing with $\Col(\omega,\lambda)$ (see Steel \cite[Corollary 4.8]{SteDMT},) and Hauser \cite[Theorem 3.14]{HauProjAbs} showed that if one-step projective generic absoluteness holds, then either there is an inner model with a Woodin cardinal or the core model $K$ has infinitely many strong cardinals.

For two-step generic absoluteness, the results of Woodin and Hauser give a level-by-level equiconsistency: for every positive integer $n$, two-step $\bfSigma^1_{n+3}$ generic absoluteness is equiconsistent with the existence of $n$ strong cardinals. Woodin showed that if there are $n$ strong cardinals $\kappa_1 < \cdots < \kappa_n$, then two-step $\bfSigma^1_{n+3}$ generic absoluteness holds after forcing with $\Col(\omega, 2^{2^{\kappa_n}})$ (see Steel \cite[Corollary 4.7]{SteDMT},) and Hauser \cite[Theorem 3.10]{HauProjAbs} showed that if two-step $\bfSigma^1_{n+3}$ generic absoluteness holds, then either there is an inner model with a Woodin cardinal or the core model $K$ has at least $n$ strong cardinals.
We will show that the cardinal $2^{2^{\kappa_n}}$ in Woodin's theorem can be reduced to $2^{\kappa_n}$:

\begin{theorem}\label{thm:sigma-1-n-plus-3}
 Let $\kappa_1 < \cdots < \kappa_n$ be strong cardinals where $n$ is a positive integer. Let $G\subset \Col(\omega,2^{\kappa_n})$ be a $V$-generic filter.  Then $V[G]$ satisfies two-step $\bfSigma^1_{n+3}$ generic absoluteness.
\end{theorem}

Note that Theorem \ref{thm:sigma-1-n-plus-3} implies Woodin's theorem because $\Col(\omega,2^{\kappa_n})$ regularly embeds into $\Col(\omega,2^{2^{\kappa_n}})$.

\begin{remark}
 Theorem \ref{thm:sigma-1-n-plus-3} is optimal in the sense that two-step (or even one-step) $\bfSigma^1_{n+3}$ generic absoluteness may fail after collapsing every cardinal less than $2^{\kappa_n}$ to be countable. Suppose that $V = K$ and there are $n$ strong cardinals. Let $\kappa_1 < \cdots < \kappa_n$ be the first $n$ strong cardinals. Then GCH holds and in particular $2^{\kappa_n} = \kappa_n^+$, so two-step $\bfSigma^1_{n+3}$ generic absoluteness holds after forcing with $\Col(\omega,\kappa_n^+)$ by Theorem \ref{thm:sigma-1-n-plus-3}, but (as explained below) it fails after forcing with only $\Col(\omega,\kappa_n)$.

 In fact, in this situation even one-step $\bfSigma^1_{n+3}$ generic absoluteness fails after forcing with $\Col(\omega,\kappa_n)$ because there is a $\bfSigma^1_{n+3}$ statement that is not absolute between generic extensions of $V$ by $\Col(\omega,\kappa_n)$ and $\Col(\omega,\kappa_n^+)$. To see this, note that the number of $\mathord{<}\kappa_n$-strong cardinals less than $\kappa_n$ is $n-1$ because every cardinal that is strong up to a strong cardinal is strong. Therefore by the proof of Hauser \cite[Theorem 3.10]{HauProjAbs}, if we take a $V$-generic filter $G \subset \Col(\omega, \kappa_n)$ and a real $x \in V[G]$ coding the extender sequence of $K$ up to $\kappa_n$, the statement ``$\omega_1$ is the cardinal successor of $\kappa_n$ in $K$'' can be expressed in a $\Pi^1_{n+3}(x)$ way. This statement is true in $V[G]$ because $V = K$, but it becomes false after further forcing to collapse $\kappa_n^+$.
\end{remark}

One might still ask whether the cardinal $2^{\kappa_n}$ in Theorem \ref{thm:sigma-1-n-plus-3} can be replaced by $\kappa_n^+$. Although this remains open in the general case, we can prove it in the case $n = 1$:

\begin{theorem}\label{thm:sigma-1-4}
 Let $\kappa$ be a strong cardinal and let $G\subset \Col(\omega,\kappa^+)$ be a $V$-generic filter. Then $V[G]$ satisfies two-step $\bfSigma^1_4$ generic absoluteness.
\end{theorem}

\begin{question}
 Let $\kappa_1 < \cdots < \kappa_n$ be strong cardinals where $n \ge 2$. Let $G \subset \Col(\omega,\kappa_n^+)$ be a $V$-generic filter.  Must $V[G]$ satisfy two-step $\bfSigma^1_{n+3}$ generic absoluteness?
\end{question}

Theorems \ref{thm:sigma-1-n-plus-3} and \ref{thm:sigma-1-4} are proved in Sections \ref{sec:proof-of-thm-sigma-1-n-plus-3} and \ref{sec:proof-of-thm-sigma-1-4} respectively. Some prior knowledge of tree representations of sets of reals would help the reader, although Section \ref{sec:tree-reps} contains a brief review of the subject. Section \ref{sec:proof-of-thm-sigma-1-4} requires some inner model theory, but nothing very technical.

\section{Tree representations of sets of reals}\label{sec:tree-reps}

In this section we briefly review some definitions regarding tree representations of sets of reals, where by ``reals'' we mean elements of the Baire space $\omega^\omega$. For more information on this topic, see Moschovakis \cite{MosDST2009} or Kechris and Moschovakis \cite{KecMosScales}. We also define a generalization of semiscales that we call \emph{proto-semiscales} (see Definition \ref{def:proto-semiscale} below.)

For a class $X$, a \emph{tree on $X$} is a subset of $X^{\mathord{<}\omega}$ that is closed under initial segments. When we consider trees on Cartesian products such as $\omega \times \Ord$, we conflate sequences of pairs with pairs of sequences. Because we require trees to be sets, a tree on $\omega \times \Ord$ is actually a tree on $\omega \times \gamma$ for some ordinal $\gamma$, but for simplicity we avoid naming $\gamma$ whenever possible.

For a tree $T$ on $\omega \times \Ord$ we let $[T]$ denote the set of branches of $T$, meaning the set of pairs $(x,f) \in \omega^\omega \times \Ord^\omega$ such that $(x\restriction n, f \restriction n) \in T$ for all $n <\omega$. Note that $[T]$ is a closed subset of $\omega^\omega \times \Ord^\omega$, and conversely every closed subset of $\omega^\omega \times \Ord^\omega$ is the set of branches of some tree on $\omega \times \Ord$. For a relation $R$ we use the notation $pR$ for the projection of $R$ along its last coordinate, so $p[T]$ is the set of all reals $x$ such that $(x,f) \in [T]$ for some $f \in \Ord^\omega$. If $A = p[T]$ then we say \emph{$T$ projects to $A$}, or $T$ \emph{is a tree representation of $A$}.

For a tree $T$ on $\omega \times \Ord$ and a real $x$, we define the \emph{section} of $T$ by $x$ as follows: $T_x = \{s \in \Ord^{\mathord{<}\omega} : (x \restriction |s|, s) \in T\}$.  Then $T_x$ is a tree on $\Ord$ and we have $x \in p[T]$ if and only if $T_x$ is illfounded. If $x \notin p[T]$ (so that $T_x$ is wellfounded) then we let $\rank_{T_x}$ denote the rank function on $T_x$.  For convenience we define $\rank_{T_x}(s)$ to be zero in the case $s \notin T_x$, so that $\rank_{T_x}$ is defined on all $s \in \Ord^{\mathord{<}\omega}$. For various finite sequences $s \in \Ord^{\mathord{<}\omega}$, the corresponding ordinals $\rank_{T_x}(s)$ provide various measures of how ``close'' $x$ is to $p[T]$.

If $T$ and $\tilde{T}$ are trees on $\omega \times \Ord$ and $\mathbb{P}$ is a poset, the pair $(T,\tilde{T})$ is called \emph{$\mathbb{P}$-absolutely complementing} if $p[\tilde{T}] = \omega^\omega \setminus p[T]$ in every generic extension of $V$ by $\mathbb{P}$. Because every generic extension of $V$ by $\mathbb{P}$ is contained in a generic extension of $V$ by $\Col(\omega,\lambda)$ where $\lambda = |\mathbb{P}|$, every $\Col(\omega,\lambda)$-absolutely complementing pair of trees is $\mathbb{P}$-absolutely complementing, so for our purposes it will suffice to consider Levy collapse posets.

For a set of reals $A$, a \emph{norm} on $A$ is a function $\varphi:A \to \Ord$. For every norm $\varphi$ on $A$, we may define a corresponding prewellordering $\le_{\varphi}$ of $A$ by $x \le_{\varphi} y \iff \varphi(x) \le \varphi(y)$. Two norms on $A$ are \emph{equivalent} if their corresponding prewellorderings are equal.

A \emph{semiscale} on $A$ is a sequence $\vec{\varphi} = (\varphi_i :i<\omega)$ of norms on $A$ such that for every sequence of reals $(x_n:n<\omega)$ in $A$ and every real $y$, if $(x_n:n<\omega)$ converges to $y$ and for each $i<\omega$ the sequence of ordinals $(\varphi_i(x_n): n<\omega)$ is eventually constant, then $y \in A$. Note that replacing each norm by an equivalent norm preserves this defining property of semiscales.

From a semiscale $\vec{\varphi} = (\varphi_i :i<\omega)$ on $A$ we may define the \emph{tree associated to $\vec{\varphi}$}, which is the tree on $\omega \times \Ord$ consisting of all finite sequences of pairs 
\[((x(0), \varphi_0(x)), \ldots, (x(n-1), \varphi_{n-1}(x)))\]
where $x \in A$ and $n < \omega$. The relevant consequence of this definition is that if $\tilde{T}$ is the tree associated to a semiscale on $A$, then $\tilde{T}$ projects to $A$. (This consequence is proved for the tree of a scale by Kechris and Moschovakis \cite[Section 6.2]{KecMosScales}, and the proof applies equally well to semiscales.)

The following generalization of semiscales will also be useful.

\begin{definition}\label{def:proto-semiscale}
 A \emph{proto-semiscale} on a set of reals $A$ is a set $\mathcal{C}$ of norms on $A$ such that for every sequence of reals $(x_n:n<\omega)$ in $A$ and every real $y$, if $(x_n:n<\omega)$ converges to $y$ and for each norm $\varphi \in \mathcal{C}$ the sequence of ordinals $(\varphi(x_n): n<\omega)$ is eventually constant, then $y \in A$.
\end{definition}

Note that replacing each norm by an equivalent norm preserves the defining property of proto-semiscales. Note also that if the sequence of norms $(\varphi_i :i<\omega)$ is a semiscale on $A$ then the set of norms $\{\varphi_i :i<\omega\}$ is a proto-semiscale on $A$, and if the set of norms $\mathcal{C}$ is a countable proto-semiscale on $A$ then every enumeration of $\mathcal{C}$ by $\omega$ is a semiscale on $A$.  If we have a proto-semiscale $\mathcal{C}$ on $A$ and we want to define an associated tree representation for $A$, we will first need to enumerate $\mathcal{C}$ by $\omega$.

\section{Proof of Theorem \ref{thm:sigma-1-n-plus-3}}\label{sec:proof-of-thm-sigma-1-n-plus-3}

The main ingredient in the proof of Theorem \ref{thm:sigma-1-n-plus-3} will be the following lemma. Woodin proved a version of this lemma with $2^{2^{\kappa_n}}$ in place of $2^{\kappa_n}$: see Steel \cite[Theorem 4.5]{SteDMT}. (That version also required the codomain $M$ of the elementary embedding $j:V \to M$ to have a bit more agreement with $V$, but this difference is not relevant to our desired application where $\kappa$ is a strong cardinal.)

\begin{lemma}\label{lem-strong-2-kappa}
 Let $\kappa$ and $\lambda$ be cardinals with $2^\kappa < \lambda$.
 Let $j : V \to M$ be an elementary embedding such that $\crit(j) = \kappa$ and $j(\kappa) > \lambda$ and $M$ is a transitive class with $\powerset(\lambda) \subset M$. Let $T$ be a tree on $\omega \times \Ord$ and let $G \subset \Col(\omega,2^\kappa)$ be a $V$-generic filter. Then in $V[G]$ there is a tree $\tilde{T}$ on $\omega \times \Ord$ such that the pair $(j(T), \tilde{T})$ is $\Col(\omega,\lambda)$-absolutely complementing.
\end{lemma}

\begin{remark}
 Woodin used a system of measures to build an absolute complement $\tilde{T}$ for $j(T)$ via a Martin--Solovay construction, and the cardinal $2^{2^\kappa}$ appeared as an upper bound on the number of measures on $\kappa$. In our proof, we will instead build a semiscale from norms corresponding to rank functions, and the cardinal $2^\kappa$ will appear as an upper bound on the number of norms required. Wilson \cite[Lemma 3.1]{WilUBGenAbs} used a similar argument. (A version of Lemma \ref{lem-strong-2-kappa} was initially included in that paper, but later removed because it threatened to take the paper on an overly long tangent.)
\end{remark}

Once Lemma \ref{lem-strong-2-kappa} is proved, Theorem \ref{thm:sigma-1-n-plus-3} will follow routinely as in Steel \cite[Section 4]{SteDMT}. The argument is identical except for the substitution of $2^{\kappa_n}$ for $2^{2^{\kappa_n}}$, so we only give a brief summary here. We begin with the Martin--Solovay tree representations of $\Pi^1_2$ sets, which exist because every set has a sharp, and may be taken to project to the intended sets in arbitrarily large generic extensions. Then we build tree representations of $\Sigma^1_3,\Pi^1_3, \Sigma^1_4, \Pi^1_4,\ldots,\Sigma^1_{n+2}, \Pi^1_{n+2}$ sets, which may also be taken to project to the intended sets in arbitrarily large generic extensions. In the end, the existence of such tree representations of $\Pi^1_{n+2}$ sets will imply two-step $\bfSigma^1_{n+3}$ generic absoluteness by a standard argument involving the absoluteness of wellfoundedness.

To go from $\Pi^1_{i+1}$ to $\Sigma^1_{i+2}$ is trivial. To go from $\Sigma^1_{i+2}$ to $\Pi^1_{i+2}$ we collapse $2^{\kappa_i}$, where $\kappa_i$ is the $i\text{th}$ strong cardinal, and apply Lemma \ref{lem-strong-2-kappa}. Because we may take $\lambda$ to be arbitrarily large in Lemma \ref{lem-strong-2-kappa}, this step preserves the property that our trees project to the intended sets in arbitrarily large generic extensions. Note that the difference between the two trees $T$ and $j(T)$ in Lemma \ref{lem-strong-2-kappa} is not a problem here. In the case $i = 1$ for example, let $G \subset \Col(\omega, 2^{\kappa_1})$ be a $V$-generic filter, let $H \subset \Col(\omega, \lambda)$ be a $V[G]$-generic filter, let $A$ be a $\Sigma^1_3$ property, and let $T$ be a tree in $V$ such that every generic extension of $V$ by a poset of cardinality at most $\lambda$ satisfies $p[T] = \{x \in \omega^\omega : A(x)\}$. Then the model $V[G][H]$ satisfies $p[T] = \{x \in \omega^\omega : A(x)\}$ and by the elementarity of $j$ the model $M[G][H]$ satisfies $p[j(T)] = \{x \in \omega^\omega : A(x)\}$.  These two models have the same reals because $\powerset(\lambda) \subset M$, so $A$ defines the same $\Sigma^1_3$ set in both models.

It therefore remains to prove Lemma \ref{lem-strong-2-kappa}. In fact, we will prove a sharper version that will be needed for the proof of Theorem \ref{thm:sigma-1-4} in Section \ref{sec:proof-of-thm-sigma-1-4}:

\begin{lemma}\label{lem-strong-eta}
 Let $\kappa$, $\eta$, and $\lambda$ be cardinals with $\kappa \le \eta < \lambda$.
 Let $j : V \to M$ be an elementary embedding such that $\crit(j) = \kappa$ and $j(\kappa) > \lambda$ and $M$ is a transitive class with $\powerset(\lambda) \subset M$. Let $T$ be a tree on $\omega \times \Ord$ with $|V_{\kappa+1} \cap L(j(T),V_\kappa)| \le \eta$ and let $G \subset \Col(\omega,\eta)$ be a $V$-generic filter. Then in $V[G]$ there is a tree $\tilde{T}$ on $\omega \times \Ord$ such that the pair $(j(T), \tilde{T})$ is $\Col(\omega,\lambda)$-absolutely complementing.
\end{lemma}

Because $\kappa$ is inaccessible we have $|V_{\kappa+1} \cap L(j(T),V_\kappa)| \le 2^\kappa$ for every tree $T$,
so Lemma \ref{lem-strong-eta} with $\eta = 2^\kappa$ implies Lemma \ref{lem-strong-2-kappa}. To complete this section it remains to prove Lemma \ref{lem-strong-eta}.

\begin{remark}
 The proof will resemble a well-known construction which, given a countable tree representation $T$ of a $\bfSigma^1_1$ set, builds a semiscale $\vec{\varphi}$ on (and thereby an associated tree representation $\tilde{T}$ of) the complementary $\bfPi^1_1$ set. The main difference here is that in Lemma \ref{lem-strong-eta} we may not assume $T$ is countable; indeed, $T$ will have cardinality at least $\kappa$ in the desired applications. Therefore what we first get is only a proto-semiscale. To get a semiscale, we will need to collapse the proto-semiscale to be countable. If the tree itself is collapsed, then it's not clear this proto-semiscale will be useful; however, by using a given degree of strongness of $\kappa$, we may ``inflate'' $T$ to $j(T)$ and show that it suffices to collapse something less than the cardinality of $j(T)$. As noted previously, the difference between $T$ and $j(T)$ will not matter for applications within the projective hierarchy.
\end{remark}

\begin{proof}[Proof of Lemma \ref{lem-strong-eta}]
 Note that for every generic extension $V[g]$ of $V$ by a poset in $V_\kappa$, the map $j$ extends to an elementary embedding $V[g] \to M[g]$, giving $p[T]^{V[g]} = p[j(T)]^{M[g]}$. Of course, $p[j(T)]^{M[g]} = p[j(T)]^{V[g]}$. Let $H \subset \Col(\omega,\lambda)$ be a $V[G]$-generic filter. By the elementarity of $j$, the fact that $p[T]^{V[g]} = p[j(T)]^{V[g]}$ for every generic extension $V[g]$ of $V$ by a poset in $V_\kappa$, and the fact that $j(\kappa)>\lambda$, we have $p[j(T)]^{M[G][H]} = p[j(j(T))]^{M[G][H]}$. The models $M[G][H]$ and $V[G][H]$ have the same reals because $\powerset(\lambda)\subset M$, so
 \[ p[j(T)]^{V[G][H]} = p[j(T)]^{M[G][H]} = p[j(j(T))]^{M[G][H]} = p[j(j(T))]^{V[G][H]}.\]
 Let $A$ denote the complement of this common set of reals:
 \[A = (\omega^\omega \setminus p[j(T)])^{V[G][H]} = (\omega^\omega \setminus p[j(T)])^{M[G][H]} = \ldots.\]
 Take an ordinal $\gamma$ such that $T$ is a tree on $\omega \times \gamma$.
 Then for each $t \in j(\gamma)^{\mathord{<}\omega}$, we may define a norm $\varphi_t$ on $A$ by
 \[ \varphi_t(x) = \rank_{j(j(T))_x}(j(t)).\]
 Recall that this rank is defined as zero if $j(t) \notin j(j(T))_x$.

\begin{claim}\label{claim:proto-semiscale}
 The set of norms $\{\varphi_t : t \in j(\gamma)^{\mathord{<}\omega}\}$ is a proto-semiscale on $A$.
\end{claim}

\begin{proof}
 Let $(x_n : n < \omega)$ be a sequence of reals in $A$ and let $y$ be a real in $V[G][H]$. Assume that $(x_n:n<\omega)$ converges to $y$ and for every  $t \in j(\gamma)^{\mathord{<}\omega}$ the sequence of ordinals $(\varphi_t(x_n): n<\omega)$ is   eventually constant; let $\lambda_t$ denote its eventual value.  We want to show $y \in A$. Suppose toward a contradiction that $y \in p[j(T)]$ as witnessed by $f \in j(\gamma)^\omega$. That is, for all $i < \omega$ we have $f \restriction i \in j(T)_y$.  Then we have $j(f \restriction i) \in j(j(T))_y$, and because $x_n \to y$ as $n \to \omega$, we have $j(f \restriction i) \in j(j(T))_{x_n}$ for all sufficiently large $n$.  Therefore for all sufficently large $n$ the norm values $\varphi_{f \restriction i}(x_n)$ and $\varphi_{f \restriction (i+1)}(x_n)$ are the ranks of a parent and child node respectively in the well-founded tree $j(j(T))_{x_n}$, so the eventual norm values satisfy the inequality $\lambda_{f \restriction i} > \lambda_{f \restriction (i+1)}$. This inequality holds for all $i < \omega$, so we get an infinite decreasing sequence of ordinals, which is a contradiction.
\end{proof}
 
We can strengthen Claim \ref{claim:proto-semiscale} as follows:

\begin{claim}
 There is a set $\sigma \subset j(\gamma)^{\mathord{<}\omega}$ in $V$ such that $|\sigma|^V \le \eta$ and 
 the set of norms $\{\varphi_t : t \in \sigma\}$ is a proto-semiscale on $A$.
\end{claim}
\begin{proof}
 For every $t \in j(\gamma)^{\mathord{<}\omega}$ we may define, in any model containing the tree $j(T)$, a norm $\psi_t$ on the set of reals $\omega^\omega \setminus p[j(T)]$ in that model by $\psi_t(x) = \rank_{j(T)_x}(t)$. Note that $\psi_t$ is defined like $\varphi_t$ but with one fewer level of applications of $j$. Define an equivalence relation $\sim$ on $j(\gamma)^{\mathord{<}\omega}$ by $t \sim t'$ if and only if in every generic extension of $V$ by a poset in $V_\kappa$ the norms $\psi_t$ and $\psi_{t'}$ are equivalent (meaning their corresponding prewellorderings $\le_{\psi_t}$ and $\le_{\psi_{t'}}$ are equal.) Note that $t \sim t'$ if and only if $S_t = S_{t'}$, where for every $t \in j(\gamma)^{\mathord{<}\omega}$ we define the set
 \[S_t =  \{(\mathbb{P}, p, \dot{x}, \dot{y}) \in V_\kappa : p \forces_{\mathbb{P}} \dot{x} \le_{\psi_t} \dot{y} \}.\]
 Because $S_t \in V_{\kappa+1} \cap L(j(T), V_\kappa)$ for all $t \in j(\gamma)^{\mathord{<}\omega}$, the number of equivalence classes of $\sim$ is at most $\eta$, so we may take a set $\sigma \subset j(\gamma)^{\mathord{<}\omega}$ in $V$ such that $|\sigma|^V \le \eta$ and for every $t \in j(\gamma)^{\mathord{<}\omega}$ there is some $t' \in \sigma$ with $t \sim t'$. Note that if $t \sim t'$, then from the elementarity of $j$ and the fact that $j(\kappa) > \lambda$ it follows that the two norms $\varphi_{t}$ and $\varphi_{t'}$ on the set of reals $(\omega^\omega \setminus p[j(j(T))])^{M[G][H]}$ (which is equal to $A$) are equivalent. Therefore every norm in the set $\{\varphi_t : t \in j(\gamma)^{\mathord{<}\omega}\}$ is equivalent to a norm in the set $\{\varphi_t : t \in \sigma\}$. Because replacing every norm by an equivalent norm preserves the defining property of proto-semiscales, the claim now follows from Claim \ref{claim:proto-semiscale}.
\end{proof}
 
 Fixing an enumeration $(t_i : i <\omega)$ of $\sigma$ in $V[G]$, the corresponding sequence of norms $\vec{\varphi} = (\varphi_{t_i} : i < \omega)$ is therefore a semiscale on $A$. Let $\tilde{T}$ be its associated tree. In $V[G][H]$ the tree $\tilde{T}$ is definable from the semiscale $\vec{\varphi}$, which in turn is definable from the tree $j(j(T)) \in V$ and the sequence of nodes $(j(t_i) : i < \omega) \in V[G]$. Because $\tilde{T}$ is a subset of $V[G]$ (in fact of $V$) and the poset $\Col(\omega,\lambda)$ is almost homogeneous, we therefore have $\tilde{T} \in V[G]$. Because $\tilde{T}$ projects to the set $A = \omega^\omega \setminus p[j(T)]$ in $V[G][H]$, it is a $\Col(\omega,\lambda)$-absolute complement for $j(T)$ in $V[G]$.
\end{proof}

\section{Proof of Theorem \ref{thm:sigma-1-4}}\label{sec:proof-of-thm-sigma-1-4}

In this section we will need to use parts of the proof of $\Sigma^1_3$ correctness of the core model in Steel \cite[Section 7D]{SteCMIP}. In particular we will need the following definition and lemma regarding a version of the Martin--Solovay tree representations of $\Sigma^1_3$ sets.

\begin{definition}
 Let $\kappa$ be an inaccessible cardinal such that every set in $V_\kappa$ has a sharp.
 \begin{enumerate}
  \item For every ordinal $\alpha \ge 1$, let $u_\alpha$ be the $\alpha^\text{th}$ uniform indiscernible for the models $L[x]$ where $x \in V_\kappa$.
  \item For every $\Sigma^1_3$ property $A$, let $T_A$ denote the Martin--Solovay tree representation
   of $A$ as it is defined in Steel \cite[Section 7D]{SteCMIP}, which is a tree on $\omega \times u_\omega$.\footnote{Steel defines the Martin--Solovay tree representations of $\Pi^1_2$ sets instead of $\Sigma^1_3$ sets, but we may use the fact that tree representations of $\Sigma^1_{n+1}$ sets are definable from tree representations of $\Pi^1_n$ sets in a local and uniform way.}
 \end{enumerate}
\end{definition}

Note that $u_1 = \kappa$ and $u_2 \le \kappa^+$. The following lemma is an adaptation of a lemma that was implicit in M.~Magidor's simplified proof of the $\Sigma^1_3$ correctness of the Dodd--Jensen core model (see Hjorth \cite[Theorem 2.1]{Hjou2}.)  Instead of sharps for reals, we use sharps for elements of $V_\kappa$ as in Steel \cite[Lemma 7.10]{SteCMIP}. What we state here is just a more abstract consequence of the argument given by Steel; the proof is identical, so we will not repeat it.

\begin{lemma}[essentially Magidor; see Steel {\cite[Lemma 7.10]{SteCMIP}}]\label{lem:magidor}
 Let $\kappa$ be an inaccessible cardinal such that every set in $V_\kappa$ has a sharp.  Let $M$ be a transitive model of ZFC (either a set or proper class in $V$) such that $\kappa \in M$ and $M$ satisfies ``every set in $V_\kappa$ has a sharp.'' If $u_2^M = u_2$ then $u_\omega^M = u_\omega$ and $T_A^M = T_A$ for every $\Sigma^1_3$ property $A$.
\end{lemma}

We may now proceed with the proof of Theorem \ref{thm:sigma-1-4}. Let $\kappa$ be a strong cardinal. We want to show that two-step $\bfSigma^1_4$ generic absoluteness holds after forcing with $\Col(\omega,\kappa^+)$. By a standard argument involving the absoluteness of wellfoundedness, a sufficient condition for two-step $\bfSigma^1_4$ generic absoluteness to hold is:
\begin{itemize}
 \item[($\ast$)] For every $\Pi^1_3$ property $B$ and every uncountable cardinal $\lambda$ there is a tree $\tilde{T}$ on $\omega \times \Ord$ such that $\emptyset \forces_{\Col(\omega,\lambda)} p[\tilde{T}] = \{x \in \omega^\omega : B(x)\}$.
\end{itemize}
In order to show that condition ($\ast$) holds after forcing with $\Col(\omega,\kappa^+)$, we will need to consider two cases.

The easy case is where $\bfDelta^1_2$ determinacy holds in every generic extension of $V$. In this case we will show that condition ($\ast$) holds already in $V$ (and therefore two-step $\bfSigma^1_4$ generic absoluteness holds already in $V$.) Let $B$ be a $\Pi^1_3$ property, let $\lambda$ be an infinite cardinal, and let $H \subset \Col(\omega,\lambda)$ be a $V$-generic filter. Work in $V[H]$. By $\bfDelta^1_2$ determinacy and the second periodicity theorem of Moschovakis \cite[Corollary 6C.4]{MosDST2009} applied to the pointclass $\Sigma^1_2$, the pointclass $\Pi^1_3$ has the scale property. In particular the $\Pi^1_3$ set $\{x \in \omega^\omega : B(x)\}$ has a definable semiscale and is therefore the projection of a definable tree $\tilde{T}$ on $\omega \times \Ord$.  Because the poset $\Col(\omega,\lambda)$ is almost homogeneous we have $\tilde{T} \in V$, witnessing condition ($\ast$) for $B$ and $\lambda$ in $V$.

The other case is where $\bfDelta^1_2$ determinacy fails in some generic extension of $V$. In this case we will need to collapse $\kappa^+$ and apply Lemmas \ref{lem:magidor} and \ref{lem-strong-eta} to the Martin--Solovay trees. Because $V_\kappa \prec_{\Sigma_2} V$ it follows in this case that $\bfDelta^1_2$ determinacy fails in some generic extension of $V$ by a poset in $V_\kappa$. Because we are ultimately interested in $V^{\Col(\omega,\kappa^+)}$, by passing to such a small generic extension we may assume without loss of generality that $\bfDelta^1_2$ determinacy fails already in $V$.

Fix a real $z_0 \in V$ such that $\Delta^1_2(z_0)$ determinacy fails in $V$. Because every set has a sharp,
we have two-step $\bfSigma^1_3$ generic absoluteness and in particular $\Sigma^1_3(z_0)$ generic absoluteness.
Therefore if we take two $\Sigma^1_2(z_0)$ properties that are complementary in $V$ and define a failure of $\Delta^1_2(z_0)$ determinacy in $V$, they remain complementary in every generic extension of $V$ and define a failure of $\Delta^1_2(z_0)$ determinacy there also.  From the failure of $\Delta^1_2(z_0)$ determinacy in every generic extension of $V$, it follows by a theorem of Woodin (see Neeman \cite[Corollary 2.3]{NeeOptimal}, which relativizes to an arbitrary real) that no generic extension of $V$ has a proper class inner model containing $z_0$ with a Woodin cardinal.

Let $B$ be a $\Pi^1_3$ property, say $B(x) \Leftrightarrow \neg A(x)$ where $A$ is $\Sigma^1_3$, and
let $T_A = T_A^V$ be the Martin--Solovay tree representation of $A$.

\begin{claim}\label{claim:trees}
 For every generic extension $V[g]$ of $V$ by a poset in $V_\kappa$ and every real $z \in V[g]$ such that $z_0 \in L[z]$, we have $V[g] \models |\omega^\omega \cap L[T_A,z]| \le \omega_1$.
\end{claim}
\begin{proof}
 Take an inner model $N$ of ZFC such that $V_\kappa \subset N$, $\kappa$ is measurable in $N$, and $N$ has a measurable cardinal $\Omega > \kappa$. (Because $\kappa$ is strong, we may take an elementary embedding $j : V \to N$ with critical point $\kappa$ where $V_{\kappa +2} \subset N$, which implies the desired properties for $N$.) Let $V[g]$ be a generic extension of $V$ by a poset in $V_\kappa$. Note that $\kappa$ and $\Omega$ remain measurable cardinals in $N[g]$.
 
 Let $z$ be a real in $V[g]$ such that $z_0 \in L[z]$ and let $K^c(z)$ be the background certified core model built over $z$ in $N[g]$ up to the measurable cardinal $\Omega$. The model $K^c(z)$ has no Woodin cardinal: otherwise we could iterate a measure on $\Omega$ through the ordinals to obtain a proper class inner model containing $z$ (and therefore also containing $z_0$) with a Woodin cardinal, contradicting our assumption. Therefore the core model $K(z)$ over $z$ in $N[g]$ up to $\Omega$ exists. By the measurability of $\kappa$ in $N[g]$ and the proof of Steel \cite[Theorem 7.9]{SteCMIP} we have $u_2^{K(z)} = u_2^{N[g]}$. Moreover, it is easy to see that $u_2^{N[g]} = u_2^{V[g]} = u_2^V$, so by Lemma \ref{lem:magidor} we have
 \[T_A^{K(z)} = T_A^{N[g]} = T_A^{V[g]} = T_A.\]
 Therefore $T_A \in K(z)$, so $L[T_A,z] \subset K(z) \subset V[g]$ and the claim follows by CH in $K(z)$.
\end{proof}

Let $\lambda > \kappa^+$ be a cardinal. Because $\kappa$ is strong, we may take an elementary embedding $j: V \to M$ such that $\crit(j) = \kappa$ and $j(\kappa) > \lambda$ and $M$ is a transitive class with $\powerset(\lambda) \subset M$.

\begin{claim}\label{claim:trees2}
 $|V_{\kappa+1} \cap L(j(T_A),V_\kappa)| \le \kappa^+$.
\end{claim}
\begin{proof}
 Take a generic extension $V[g]$ of $V$ by $\Col(\omega,\kappa)$ and take a real $z \in V[g]$ coding $V_\kappa$. Because $z$ is in a generic extension of $M$ by a poset of cardinality less than $j(\kappa)$ and $z_0 \in L[z]$, by Claim \ref{claim:trees} and the elementarity of $j$ we have $M[g] \models |\omega^\omega \cap L[j(T_A),z]| \le \omega_1$. Note that $\omega_1^{M[g]} = \kappa^+$.  Because $L(j(T_A),V_\kappa) \subset L[j(T_A),z]$ and $V_\kappa$ is countable in $L[j(T_A),z]$, every subset of $V_\kappa$ in $L(j(T_A),V_\kappa)$ produces a different real in $L[j(T_A),z]$ and the claim follows.
\end{proof}

Now by Lemma \ref{lem-strong-eta} with $\eta = \kappa^+$, if we let $G \subset \Col(\omega,\kappa^+)$ be a $V$-generic filter then in $V[G]$ there is a tree $\tilde{T}$ on $\omega \times \Ord$ such that the pair $(j(T_A), \tilde{T})$ is $\Col(\omega,\lambda)$-absolutely complementing. Let $H \subset \Col(\omega,\lambda)$ be a $V[G]$-generic filter. Note that every generic extension of $V$ by a poset in $V_\kappa$ satisfies $p[T_A] = \{x \in \omega^\omega : A(x)\}$ (this follows from Lemma \ref{lem:magidor} because small forcing preserves $u_2$.) Therefore by the elementarity of $j$ and the fact that $j(\kappa) > \lambda$ it follows that $M[G][H]$ satisfies $p[j(T_A)] = \{x \in \omega^\omega : A(x)\}$.

The models $M[G][H]$ and $V[G][H]$ have the same reals because $\powerset(\lambda)\subset M$, so $V[G][H]$ also satisfies $p[j(T_A)] = \{x \in \omega^\omega : A(x)\}$. Bcause the pair $(j(T_A), \tilde{T})$ projects to complements in $V[G][H]$, it follows that $V[G][H]$ satisfies $p[\tilde{T}] = \{x \in \omega^\omega : B(x)\}$. Therefore the tree $\tilde{T}$ witnesses condition ($\ast$) for $B$ and $\lambda$ in $V[G]$. The proof of Theorem \ref{thm:sigma-1-4} is complete.

\section{Acknowledgments}

The author thanks Paul Larson and Menachem Magidor for their helpful comments.


\end{document}